\documentclass[preprint,11pt]{elsarticle}
\usepackage{amssymb}
\usepackage{amsmath}
\usepackage{amsthm}
\usepackage{amsfonts}
\usepackage{color}
\newtheorem{theorem}{Theorem}
\newtheorem{proposition}[theorem]{Proposition}
\newtheorem{corollary}[theorem]{Corollary}
\theoremstyle{definition}
\newtheorem{definition}{Definition}
\newtheorem{example}{Example}
\newtheorem{remark}{Remark}
\newtheorem{lemma}{Lemma}

\allowdisplaybreaks[1]
\begin{document}

\begin{frontmatter}

\title{Exponential Splitting for Nonautonomous Linear Discrete-Time Systems in Banach Spaces}

\author[adresa,adresaa]{Mihail Megan}
\ead{megan@math.uvt.ro}

\author[adresaalba]{Ioan-Lucian Popa}
\ead{lucian.popa@uab.ro}

\address[adresa]{Academy of Romanian Scientists, Independen\c tei
54, 050094 Bucharest, Romania}
\address[adresaa]{Department of Mathematics,
Faculty of Mathematics and Computer Science, West University of
Timi\c soara, V.  P\^ arvan Blv. No. 4, 300223-Timi\c soara,
Romania}
\address[adresaalba]{Department of Mathematics, Faculty of Exact Sciences and Engineering, University "1 Decembrie" 1918 of Alba Iulia, Alba Iulia, 510009, Romania}

\begin{abstract}
In this paper we consider some concepts of exponential splitting for
nonautonomous linear discrete-time systems. These concepts are
generalizations of some well-known concepts of (uniform and
nonuniform) exponential dichotomies. Connections between these
concepts are presented and some illustrating examples prove that
these are distinct.
\end{abstract}
\end{frontmatter}
\section{Introduction}

The notion of exponential dichotomy introduced by O. Perron for
differential equations in \cite{perron} and by Ta Li \cite{tali} for
difference equations plays a central role in a large part of the
theory of dynamical systems.

The notion of dichotomy for differential equations has gained
prominence since the appearance of two fundamental monographs of
J.L. Massera,  J.J. Sch$\ddot{a}$ffer \cite{massera} and J.L.
Daleckii, M.G. Krein \cite{daleckii}. These were followed by the
important book ok W.A. Coppel \cite{coppel} who synthesized and
improved the results that existed in the literature up to 1978.

The interest in the counterpart results in difference equations
appeared in the paper of C.V. Coffman and J.J. Sch$\ddot{a}$ffer
\cite{coffman} and later, in 1981 when D. Henry included discrete
dichotomies in his book \cite{henry}. This was followed by the
classical monographs due to R.P. Agarwal \cite{agarwal} where the
dichotomy properties of discrete-time systems are studied.
Significant work was reported by C.  P$\ddot{o}$tzsche in
\cite{potche}. Notable contributions in dichotomy theory of
discrete-time systems has been also obtained in
(\cite{berezanski},\cite{huy},\cite{huy2},\cite{papashi},\cite{pinto},\cite{popa},\cite{megan},\cite{sasu},\cite{zhang}).

The most important dichotomy concept used in the qualitative theory
of dynamical systems is the uniform exponential dichotomy. In some
situations, particularly in the nonautonomous setting, the concept
of uniform exponential dichotomy is too restrictive and it is
important to consider more general behaviors.

Two different perspectives can be identified to generalize the
concept of uniform exponential dichotomy, one can define dichotomies
that depends on the initial time (and therefore are nonuniform) and,
 on the other hand, one can consider growth rates which do not imply an
exponential dichotomy behavior, in particular exponential splitting.

The first approach leads to concepts of nonuniform exponential
(respectively polynomial) dichotomies for difference equations and
can be found in the works of L. Barreira, C. Valls (\cite{bareira2},
\cite{bareira3}), A. Bento, C. Silva (\cite{bento1}, \cite{bento2},)
and L. Barreira, M. Fan, C. Valls and Z. Jimin \cite{barreira1}.

The second approach is presented in the papers of  B. Aulbach, J.
Kalbrenner \cite{aulbach1}, B. Aulbach S. Siegmund \cite{aulbach2}.

In this paper we consider two concepts of exponential splitting for
linear discrete-time systems in Banach spaces. These concepts use
two ideas of projections sequences: invariant and strongly
invariant for the respective discrete-time system, (although, in case
of invertible systems, they are equivalent). These two types of
projections sequences are distinct even in the finite dimensional
case. For each of these concepts (exponential splitting and strong
exponential splitting) we consider three important particular cases:
uniform exponential splitting, exponential dichotomy and uniform
exponential dichotomy respectively, uniform strong exponential
splitting, strong exponential dichotomy, and uniform strong
exponential dichotomy. We give characterizations of these concepts
and present connections (implications and counterexamples) between
them.

We note that we consider difference equations whose right-hand sides
are not supposed to be invertible and the splitting concepts studied
in this paper use the evolution operators in forward time. The study of noninvertible systems is of great importance and in this sense we point out the paper of B. Aulbach and J. Kalkbrenner \cite{aulbach1}, where is introduced the notion of exponential forward splitting, motivated by the fact that there are differential equations whose backward solutions are not guaranteed to exist. This approach is of interest in applications, see for example, dynamical systems generated by random parabolic equations, are not invertible (for more details see L. Zhou et al. \cite{zhou}). Also, considering asymptotic rates of the form $e^{c\rho (n)},$ where $\rho :\mathbb{N}\to\mathbb{R}$ is an increasing function, which thus may correspond to infinite Lyapunov exponents, we obtain a concept of nonuniform exponential splitting which does not assume exponential boundedness of the splitting projections, and not only  the usual exponential behavior  with $\rho (n)=n.$ For more details regarding the arbitrary growth rates we may refer to \cite{barreira4}. Also, we prove that in the particular case when the splitting projections are exponentially bounded then the two splitting concepts presented in this paper are equivalent.

\section{Preliminaries}
Let $X$ be a Banach space and ${B}(X)$ the Banach space of all
bounded linear operators on $X.$ The norms on $X$ and on ${B}(X)$
will be denoted by $\parallel\cdot\parallel .$ The identity operator
on $X$ is denoted by $I.$ If $A\in{B}(X)$ then we shall denote by
$Ker\;A$ the kernel of $A$ i.e.
\begin{equation*}
Ker\;A=\{x\in X\;\text{with}\;Ax=0\}
\end{equation*}
respectively
\begin{equation*}
Range\;A=\{Ax\;\text{with}\;x\in X\}.
\end{equation*}
We also denote by $\Delta$ the set of all pairs of all natural
numbers $(m,n)$ with $m\geq n$ i.e.
$$\Delta=\{(m,n)\in\mathbb{N}^2\;\text{with}\;m\geq n\}.$$
We also consider
$$T=\{(m,n,p)\in\mathbb{N}^3\;\text{with}\;m\geq n\geq p\}.$$

We consider the linear discrete-time system
\begin{equation*}\tag{$\mathfrak{A}$}\label{A}
x_{n+1}=A_{n}x_{n},
\end{equation*}
where $(A_n)$ is a sequence in ${B}(X).$ We associate to the system
(\ref{A}) the map
 \begin{equation*}\label{eqAmn}
{A}_m^n=\left\{
\begin{array}{l l}
A{_{m-1}}\cdot \ldots\cdot  A_{n}, &\;\;\text{if}\; m > n\\
I, &\;\;\text{if}\; m=n\\
\end{array}\right.
 \end{equation*}
 which is called the {\it evolution operator associated to
 (\ref{A}).}

 It is obvious that
\begin{equation*}
{A}_m^n{A}_n^p={A}_m^p,\;\;\text{for all}\;\;(m,n,p)\in T
\end{equation*}
and every solution of (\ref{A}) satisfies
\begin{equation*}
x_m={A}_m^nx_n\;\;\text{for all}\;\;(m,n)\in\Delta.
\end{equation*}
If for every $n\in\mathbb{N}$ the operator $A_n$ is invertible then
the system (\ref{A}) is called {\it reversible.}
\begin{definition}\label{D: Proj}
A sequence $P:\mathbb{N}\rightarrow{B}(X)$ is called a {\it
projections sequence} if
$$P(n)^2=P(n),\;\;\text{for every}\;\;n\in\mathbb{N}.$$
\end{definition}
In what follows we denote $P(n)=P_n$ for every $n\in\mathbb{N}.$
\begin{remark}
If $P$ is a projections sequence then $Q=I-P$ is also a
projections sequence (which is called {\it the complementary
projections sequence of $P$}) with
$$Ker\; Q_n=Range\; P_n\;\text{and}\;Range\; Q_n=Ker\; P_n$$
 for every $n\in\mathbb{N},$ where $Q_n=Q(n).$
\end{remark}
\begin{definition}\label{D: comp}
A projection sequence $P$ is called {\it invariant} for the system
(\ref{A}) if
$$A_nP_n=P_{n+1}A_n,\;\;\text{for all}\;\;n\in\mathbb{N}.$$
\end{definition}
\begin{remark}
If $P$ is invariant for (\ref{A}) then its complementary $Q$ is
invariant for (\ref{A}). Furthermore we have
$${A}_m^nP_n=P_m{A}_m^n\;\;\text{and}\;\;{A}_m^nQ_n=Q_m{A}_m^n$$
for all $(m,n)\in\Delta.$
\end{remark}
\begin{remark}
If $P$ is invariant for (\ref{A}) then
$${A}_m^n(Ker\;P_n)\subset Ker\;P_m\;\;\text{and}\;\;{A}_m^n(Range\;P_n)\subset Range\;P_m$$
for all $(m,n)\in\Delta.$
\end{remark}
\begin{definition}\label{D: P-expb}
A projections sequence $P$ is called {\it exponentially bounded}
if there are $M,p\geq 1$ such that
$$\parallel P_n\parallel\leq Mp^n,\;\;\text{for every}\;\;n\in\mathbb{N}$$
In the particular case when $p=1,$ $P$ is called {\it bounded.}
\end{definition}
\begin{remark}
A projections sequence $P$ is exponentially bounded if and only if
there are $M,\omega\geq 1$ such that
\begin{equation*}
\parallel P_n\parallel\leq
Me^{\omega n},\;\;\text{for all}\;\;n\in\Bbb N.
\end{equation*}
\end{remark}
\begin{lemma}\label{L: 1}
Let $P$ and $R$ be two projections sequences with complementary
$Q$ respectively $S$ and with the property
$$Range\;P_n=Range\;R_n,\;\;\text{for all}\;\;n\in\mathbb{N}.$$
Then
\begin{description}

\item[{\it{($r_1$)}}] $P_nR_n=R_n;$

\item[{\it{($r_2$)}}] $R_nP_n=P_n;$

\item[{\it{($r_3$)}}] $Q_nS_n=Q_n=(I+R_n-P_n)S_n;$

\item[{\it{($r_4$)}}] $S_nQ_n=S_n=(I+P_n-R_n)Q_n;$
\end{description}
for all $n\in\mathbb{N}.$
\end{lemma}
\begin{proof}
$(r_1)$ If $x\in X$ and $n\in\mathbb{N}$ then there is $x_0\in X$
with $R_n x=P_n x_0.$ Then $P_nR_nx=P_n^2x_0=P_nx_0=R_nx_0.$

$(r_2)$ It follows from $(r_1).$

$(r_3)$ We observe that
\begin{align*}
(I+R_n-P_n)S_n &=(R_n+Q_n)S_n=Q_nS_n=(I-P_n)(I-R_n)=\\
&=I-P_n=Q_n.
\end{align*}

$(r_4)$ It follows from $(r_3)$ by changing $P_n$ with $R_n.$
\end{proof}
\begin{definition}\label{D: P-strinv}
A sequence of projections $P$ is called {\it strongly invariant}
for the system (\ref{A}) if $P$ is invariant for (\ref{A}) and for
all $(m,n)\in\Delta$ the restriction of ${A}_m^n$ at $Ker\;P_n$ is
an isomorphism from $Ker\;P_n$ to $Ker\;P_m.$
\end{definition}
\begin{remark}
If the projections sequence $P$ is invariant for the reversible
system (\ref{A}) then it is also strongly invariant for (\ref{A}).

Indeed, if $P$ is invariant for the reversible system (\ref{A})
then ${A}_m^n$ is injective and
$${A}_m^n(Ker\;P_n)\subset Ker\;P_m,\;\;\text{for all}\;\;(m,n)\in\Delta.$$
Moreover, for every $y\in Ker\;P_m$ we have that
$$y={A}_m^n\left({A}_m^n\right)^{-1}y\;\;\text{with}\;\;x=\left({A}_m^n\right)^{-1}y\in Ker\;P_n.$$
Thus ${A}_m^n$ is surjective, which implies that $P$ is strongly
invariant for (\ref{A}).
\end{remark}
\begin{remark}
If the projections sequence $P$ is strongly invariant for the system
(\ref{A}) then there exists ${B}:\Delta\rightarrow{B}(X),$
${B}(m,n)={B}_m^n$ such that ${B}_m^n$ is an isomorphism from
$Ker\;P_m$ to $Ker\;P_n$ and
\begin{description}

\item[{\it{($b_1$)}}] ${A}_m^n{B}_m^nQ_m=Q_m;$

\item[{\it{($b_2$)}}] ${B}_m^n{A}_m^nQ_n=Q_n;$
\end{description}
for all $(m,n)\in\Delta.$

The application ${B}$ is called {\it the skew-evolution operator}
associated to the pair $(\text{\ref{A}},P).$
\end{remark}
\begin{remark}
If the projections sequence $P$ is invariant for the reversible
system (\ref{A}) then it is strongly invariant for (\ref{A}) and
the skew-evolution operator associated to the pair
$(\text{\ref{A}},P)$ is
$${B}_m^n=\left({A}_m^n\right)^{-1},\;\;\text{for all}\;\;(m,n)\in\Delta.$$
\end{remark}
For nonreversible systems there are invariant projections sequences
which are not strongly invariant. This fact is illustrated by
\begin{example}\label{E: 11}
Let $X=\Bbb{R}^{3}$ and let $(P_n)$ be the projections sequence
defined by
\begin{equation*}
{P_n}(x_{1},x_{2},x_3) =\left\{
\begin{array}{l l}
(x_1,x_2,0) &\;\;\text{if}\;\; n=0\\
\left(x_{1}+\frac{x_2}{2^{n-1}},0,0\right) &\;\;\text{if}\;\;n\geq
1.\\
\end{array}\right.
\end{equation*}
Let ({\ref{A}}) be the linear discrete-time system generated by the
sequence
$$A_n(x_1,x_2,x_3)=(2x_1,a_nx_2,4x_3) $$
where
\begin{equation*}
a_n=\left\{
\begin{array}{l l}
0 &\;\;\text{if}\;\; n=0\\
    4 & \;\;\text{if}\;\;n\geq 1\\
\end{array}\right.
\;\;\text{and}\;\;x=(x_1,x_2,x_3)\in X.
\end{equation*}
It is easy to see that the evolution operator associated to system
(\ref{A}) is given by
\begin{equation*}
{A}_m^n(x_1,x_2,x_3)=\left\{
\begin{array}{l l}
 \left(2^mx_1,0,4^mx_3\right) &\;\;\text{if}\;\; m>n=0\\
    \left(2^{m-n}x_1,4^{m-n}x_2,4^{m-n}x_3\right) & \;\;\text{if}\;\;m>n\geq 1\\
    (x_1,x_2,x_3) &\;\;\text{if}\;\; m=n\\
\end{array}\right.
\end{equation*}
 for all $(m,n)\in\Delta$ and all $x=(x_1,x_2,x_3)\in X.$ We observe that
\begin{equation*}
 {A}_nP_n(x_1,x_2,x_3) =P_{n+1}A_n(x_1,x_2,x_3)=\left\{
\begin{array}{l l}
  \left(2x_1,0,0\right) &\;\;\text{if}\;\; n=0\\
    \left(2x_1+2^{2-n}x_2,0,0\right) & \;\;\text{if}\;\;n\geq 1\\
\end{array}\right.
\end{equation*}
and hence $(P_n)$ is invariant for (\ref{A}). It is not strongly
invariant because ${A}_1^0$ is not an isomorphism from $Ker\; P_0$
to $Ker\; P_1.$

Indeed, we observe that for $y=(1,-1,0)\in Ker\;P_1$ we have that
$$
y=(1,-1,0)\neq (2x_1,0,4x_3)={A}_1^0(x_1,x_2,x_3)
$$
 for all
$x=(x_1,x_2,x_3)\in Ker\;P_0.$
\end{example}
\begin{lemma}\label{L: 22}
If the projections sequence $P$ is strongly invariant for the
system (\ref{A}) then the skew-evolution operator associated to
the pair $(\text{\ref{A}},P)$ has the following properties
\begin{description}

\item[{\it{($b_3$)}}] $Q_n{B}_m^nQ_m={B}_m^nQ_m;$

\item[{\it{($b_4$)}}]$Q_m{B}_m^mQ_m={B}_m^mQ_m=Q_m;$

\item[{\it{($b_5$)}}]${B}_m^pQ_m={B}_n^p{B}_m^pQ_m;$
\end{description}
for all $(m,n,p)\in T.$
\end{lemma}
\begin{proof}
{\it ($b_3$)} We observe that for all $(m,n,x)\in\Delta\times X$
we have that
$$Q_mx\in Range\;Q_m=Ker\;P_m$$
hence
$${B}_m^nx\in Ker\;P_n=Range\;Q_n$$
which implies
$$Q_m{B}_m^nQ_mx={B}_m^nQ_mx.$$

{\it ($b_4$)}
$Q_m{B}_m^nQ_n\stackrel{(b_3)}{=}{B}_m^mQ_n\stackrel{(b_1)}{=}Q_m.$

{\it ($b_5$)} If $(m,n,p)\in T$ then
\begin{align*}
{B}_m^pQ_m
&\stackrel{(b_3)}{=}Q_p{B}_m^pQ_m\stackrel{(b_2)}{=}{B}_n^p{A}_n^pQ_p{B}_m^pQ_m={B}_m^pQ_n{A}_n^pQ_p{B}_m^pQ_m=\\
&\stackrel{(b_2)}{=}{B}_n^p{B}_m^n{A}_m^nQ_n{A}_n^pQ_p{B}_m^nQ_m={B}_n^p{B}_m^n{A}_m^pQ_p{B}_m^pQ_m=\\
&\stackrel{(b_3)}{=}{B}_n^p{B}_m^p{A}_m^p{B}_m^pQ_m\stackrel{(b_1)}{=}{B}_n^p{B}_m^nQ_n.
\end{align*}
\end{proof}
\section{Exponential splitting with invariant projections}

In this section we consider a projections sequence $P$ which is
invariant for the system (\ref{A}). We shall denote by $Q$ the
complementary of $P.$
\begin{definition}\label{D: ES}
We say that the linear discrete-time system (\ref{A}) admits an {\it
exponential splitting (e.s.)} if there exist a projections sequence
$P$ invariant for (\ref{A}) and the constants $0<a<b,$ $N,c\geq 1$
such that
\begin{equation}\tag{$es_1$}\label{es1}
\parallel{A}_m^nP_nx\parallel\leq Nc^na^{m-n}\parallel
P_nx\parallel
\end{equation}
\begin{equation}\tag{$es_2$}\label{es2}
b^{m-n}\parallel Q_nx\parallel\leq N
c^m\parallel{A}_m^nQ_nx\parallel
\end{equation}
for all $(m,n,x)\in\Delta\times X.$ The constants $a$ and $b$ are
called the {\it growth rates} of (\ref{A}).
\end{definition}
If the system (\ref{A}) admits an exponential splitting with
\begin{description}

\item[{\it{(i)}}] $c=1$ then we say that (\ref{A}) admits an {\it
uniformly exponentially splitting (u.e.s.)};

\item[{\it{(ii)}}] $0<a<1<b$ then we say that (\ref{A}) is {\it
exponentially dichotomic (e.d.)};

\item[{\it{(iii)}}] $0<a<1<b$ and $c=1$ then we say that (\ref{A})
is {\it uniformly exponentially dichotomic (u.e.d.)};
\end{description}
\begin{remark}
The system (\ref{A}) admits an exponential splitting if and only if
there exist a projections sequence $P$ invariant for (\ref{A}) and
four real constants $\alpha <\beta,$ $\gamma \geq 0$ and $N\geq 1$
such that
\begin{equation}\tag{$es_1^{'}$}\label{es1'}
\parallel{A}_m^nP_nx\parallel\leq Ne^{\gamma n}e^{\alpha({m-n})}\parallel
P_nx\parallel
\end{equation}
\begin{equation}\tag{$es_2^{'}$}\label{es2'}
e^{\beta({m-n})}\parallel Q_nx\parallel\leq N e^{\gamma
m}\parallel{A}_m^nQ_nx\parallel
\end{equation}
for all $(m,n,x)\in\Delta\times X.$
\end{remark}
For the particular case of exponential dichotomy we have
\begin{proposition}\label{P: ed}
The system (\ref{A}) is exponentially dichotomic if and only if
there are three constants $N,c\geq 1$ and $d\in (0,1)$ such that
\begin{equation}\tag{$ed_1$}\label{ed1}
\parallel{A}_m^nP_nx\parallel\leq N c^nd^{m-n}\parallel
P_nx\parallel
\end{equation}
\begin{equation}\tag{$ed_2$}\label{ed2}
\parallel Q_nx\parallel\leq N
c^md^{m-n}\parallel{A}_m^nQ_nx\parallel
\end{equation}
for all $(m,n,x)\in\Delta\times X.$
\end{proposition}
\begin{proof}
{\it Necessity.} If (\ref{A}) is (e.d.) then there are a
projections sequence $P$ invariant for (\ref{A}) and constants
$0<a<1<b,$ $N,c\geq 1$ such that the inequalities are satisfied.
If we denote by $d=\min \{a,\frac{1}{b}\}$ then $d\in (0,1)$ and
$$\parallel{A}_m^nP_nx\parallel\leq N c^na^{m-n}\parallel P_nx\parallel\leq Nc^nd^{m-n}\parallel P_nx\parallel$$
$$\parallel Q_nx\parallel\leq Nc^mb^{-({m-n})}\parallel{A}_m^nQ_nx\parallel\leq N c^md^{m-n}\parallel{A}_m^nQ_nx\parallel$$
for all $(m,n,x)\in\Delta\times X.$

{\it Sufficiency.} It is immediate.
\end{proof}
\begin{proposition}\label{P: A-inj}
If the system (\ref{A}) admits an exponential splitting then there
exists a projections sequence $P$ invariant for (\ref{A}) such
that for every $(m,n)\in\Delta$ the restriction of ${A}_m^n$ to
$Ker\;P_n$ is injective.
\end{proposition}
\begin{proof}
If system (\ref{A}) admits an (e.s.)  with projections sequence $P$
and $x\in Ker\;P_n\bigcap Ker\;{A}_m^n$ then by Definition \ref{D:
ES}, we obtain
$$\parallel x\parallel =\parallel Q_nx\parallel\leq Nb^{n-m}c^m\parallel{A}_m^nQ_nx\parallel=Nb^{n-m}c^m\parallel Q_m{A}_m^nx\parallel=0$$
hence $x=0.$
\end{proof}
For the case of reversible systems we can give a necessary and
sufficient condition for (e.s.) by
\begin{theorem}\label{T: esvsed}
The reversible system (\ref{A}) admits an exponential splitting if
and only if there are a projections sequence $P$ invariant for
(\ref{A}) and the constants $0<a<b,$ $N,c\geq 1$ such that
\begin{equation}\tag{$res_1$}\label{res1}
\parallel{A}_m^nP_nx\parallel\leq N c^na^{m-n}\parallel
P_nx\parallel
\end{equation}
\begin{equation}\tag{$res_2$}\label{res2}
b^{m-n}\parallel \left({A}_m^n\right)^{-1}Q_mx\parallel\leq N
c^n\parallel Q_nx\parallel
\end{equation}
for all $(m,n,x)\in\Delta\times X.$
\end{theorem}
\begin{proof}
It is sufficient to prove the equivalence
(\ref{es2})$\Longleftrightarrow$(\ref{res2}). If (\ref{es2}) holds
then
\begin{align*}
b^{m-n}\parallel\left({A}_m^n\right)^{-1}Q_mx\parallel
&=b^{m-n}\parallel
Q_n\left({A}_m^n\right)^{-1}x\parallel\\
&\leq
Nc^m\parallel{A}_m^nQ_n\left({A}_m^n\right)^{-1}x\parallel\\
&=Nc^m\parallel Q_nx\parallel
\end{align*}
 for all $(m,n,x)\in\Delta\times X.$

 Conversely, from (\ref{res2}) it results
 \begin{align*}
b^{m-n}\parallel Q_nx\parallel &=b^{m-n}\parallel
\left({A}_m^n\right)^{-1}Q_m{A}_m^nQ_nx\parallel\\
&\leq
Nc^n\parallel Q_m{A}_m^nQ_nx\parallel\\
&=Nc^n\parallel {A}_m^nQ_nx\parallel
\end{align*}
for every $(m,n,x)\in\Delta\times X.$
\end{proof}
\begin{theorem}\label{T: projPR}
Let $P$ and $R$ be two projections sequences with complementarily
$Q$ and $S.$ Let $P$ and $Q$ be exponentially bounded and
$Range\;P_n=Range\;R_n$ for every $n\in\mathbb{N}.$ If system
(\ref{A}) admits an exponential splitting with projections
sequence $P$ then it also admits an exponential splitting with
respect to $R.$
\end{theorem}
\begin{proof}
Let $M,p\geq 1$ be two constants such that
$$\parallel P_n\parallel+\parallel R_n\parallel\leq Mp^n$$
for every $n\in\mathbb{N}.$ Assume that (\ref{A}) admits an (e.s.)
with projections sequence $P.$ Then, by Definition \ref{D: ES} and
Lemma \ref{L: 1}, we obtain
\begin{align*}
\parallel{A}_m^nR_nx\parallel
&=\parallel{A}_m^nP_nR_nx\parallel\leq
Na^{m-n}c^n\parallel R_nx\parallel\\
&\leq MNa^{m-n}(cp)^{n}\parallel x\parallel
=N_1a^{m-n}c_1^n\parallel x\parallel
\end{align*}
and
\begin{align*}
b^{m-n}\parallel S_nx\parallel &\leq b^{m-n}\parallel
I+P_n-R_n\parallel\cdot\parallel Q_nx\parallel\\
&\leq 2Mb^{m-n}p^n\parallel Q_nx\parallel\leq
2MNp^nc^m\parallel{A}_m^nQ_nx\parallel\\
&=2MNp^nc^m\parallel{A}_m^n(I+R_n+P_n)S_nx\parallel\\
&\leq 4M^2Np^{m+n}c^{m}\parallel{A}_m^nS_nx\parallel\\
&= N_1c_1^m\parallel{A}_m^nS_nx\parallel,
\end{align*}
for all $(m,n,x)\in\Delta\times X,$ where $N_1=4M^2N$ and
$c_1=p^2c.$
\end{proof}
\section{Exponential splitting with strongly invariant projections}
In this section we consider the particular case of exponential
splitting with projections sequence strongly invariant for a linear
discrete-time system.

Let $P:\mathbb{N}\rightarrow B(X)$ be a projections sequence
strongly invariant for the system (\ref{A}) and let
${B}:\Delta\rightarrow B(X),$ ${B}(m,n)={B}_m^n$ be the
skew-evolution operator associated to the pair of (\ref{A},P).
\begin{theorem}\label{T: esinv}
The system (\ref{A}) admits an exponential splitting with the
projections sequence $P$ if and only if there are $0<a<b$ and
$N,c\geq 1$ such that
\begin{equation}\tag{$es_1^{"}$}\label{es1"}
\parallel{A}_m^nP_nx\parallel\leq Nc^na^{m-n}\parallel
P_nx\parallel
\end{equation}
\begin{equation}\tag{$es_2^{"}$}\label{es2"}
b^{m-n}\parallel {B}_m^nQ_mx\parallel\leq N c^m\parallel
Q_mx\parallel
\end{equation}
for all $(m,n,x)\in\Delta\times X.$
\end{theorem}
\begin{proof}
We have only to prove the equivalence
(\ref{es2})$\Longleftrightarrow$(\ref{es2"}.)

{\it Necessity.} We observe that from (\ref{es2}), ($b_1$) and
$(b_3)$ we obtain
\begin{align*}
b^{m-n}\parallel{B}_m^nQ_mx\parallel
&\stackrel{(b_3)}{=}b^{m-n}\parallel Q_n{B}_m^nQ_mx\parallel\leq
Nc^m\parallel{A}_m^nQ_n{B}_m^nQ_mx\parallel=\\
&=Nc^m\parallel
Q_m{A}_m^n{B}_m^nQ_mx\parallel\stackrel{(b_1)}{=}Nc^m\parallel
Q_mx\parallel
\end{align*}
for every $(m,n,x)\in\Delta\times X.$

{\it Sufficiency.} Similarly, from $(b_2)$ and (\ref{es2"}) it
results
\begin{align*}
b^{m-n}\parallel Q_nx\parallel
&\stackrel{(b_2)}{=}b^{m-n}\parallel{B}_m^n{A}_m^nQ_nx\parallel=b^{m-n}\parallel{B}_m^nQ_m{A}_m^nx\parallel\\
&\leq Nc^m\parallel
Q_m{A}_m^nx\parallel=Nc^m\parallel{A}_m^nQ_nx\parallel
\end{align*}
for all $(m,n,x)\in\Delta\times X.$
\end{proof}
\begin{corollary}\label{C: UES}
The linear discrete-time system (\ref{A}) admits a uniform
exponential splitting with projections sequence $P$ (strongly
invariant for (\ref{A})) if and only if there exist $0<a<b$ and
$N,c\geq 1$ such that
\begin{equation}\tag{$ues_1$}\label{ues1}
\parallel{A}_m^nP_nx\parallel\leq Na^{m-n}\parallel
P_nx\parallel
\end{equation}
\begin{equation}\tag{$ues_2$}\label{ues2}
b^{m-n}\parallel {B}_m^nQ_mx\parallel\leq N \parallel
Q_mx\parallel
\end{equation}
for all $(m,n,x)\in\Delta\times X.$
\end{corollary}
Now we introduce a new concept of exponential splitting by
\begin{definition}\label{D: SES}
We say that the system (\ref{A}) admits a {\it strong exponential
splitting (s.e.s.)} if there exist a projections sequence $P$
strongly invariant for (\ref{A}) and the constants $0<a<b,$ $N\geq
1$ such that
\begin{equation}\tag{$ses_1$}\label{ses1}
\parallel{A}_m^nP_nx\parallel\leq Nc^na^{m-n}\parallel
x\parallel
\end{equation}
\begin{equation}\tag{$ses_2$}\label{ses2}
b^{m-n}\parallel {B}_m^nQ_mx\parallel\leq Nc^m \parallel
x\parallel
\end{equation}
for all $(m,n,x)\in\Delta\times X.$
\end{definition}
For the particular case $c=1,$ we say that system (\ref{A}) admits
a {\it uniform exponential splitting (u.e.s.)}.

The particular cases $0<a<1<b$ respectively $0<a<1<b$ and $c=1$
leads to the notions of {\it strong exponential dichotomy
(s.e.d.)} respectively {\it uniform strong exponential dichotomy
(u.s.e.s.)}.
\begin{remark}
The system (\ref{A}) admits a strong exponential dichotomy if and
only if there are a projections sequence $P$ strongly invariant for
(\ref{A}) and the constants $0<a<b$ and $N,c\geq 1$ such that
\begin{equation}\tag{$ses_1^{'}$}\label{ses1'}
\parallel{A}_m^nP_n\parallel\leq Nc^na^{m-n}
\end{equation}
\begin{equation}\tag{$ses_2^{'}$}\label{ses2'}
b^{m-n}\parallel {B}_m^nQ_m\parallel\leq Nc^m
\end{equation}
for all $(m,n)\in\Delta.$
\end{remark}
\begin{remark}
The system (\ref{A}) admits a strong exponential splitting if and
only if there are a projections sequence $P$ strongly invariant for
(\ref{A}) and four real constants $\alpha <\beta,$ $\gamma>0$ and
$N\geq 1$ such that
$$
\parallel{A}_m^nP_n\parallel\leq Ne^{\gamma n}e^{\alpha(m-n)}
$$
$$ e^{\beta(m-n)}\parallel
{B}_m^nQ_m\parallel\leq Ne^{\gamma m}
$$
for all $(m,n)\in\Delta.$
\end{remark}
For the particular case of strong exponential dichotomy we have
\begin{remark}
The system (\ref{A}) is strongly exponentially dichotomic if and
only if there exist a projections sequence $P$ strongly invariant
for (\ref{A}) and three constants $N,c\geq 1$ and $d\in (0,1)$ such
that
$$
\parallel{A}_m^nP_n\parallel\leq Nc^{n}d^{m-n}\parallel
P_n\parallel
$$
$$\parallel
Q_n\parallel\leq Nc^{m}d^{m-n}\parallel A_m^nQ_n\parallel
$$
for all $(m,n)\in\Delta.$
\end{remark}

A connection between (s.e.s.) and (e.s.) presents the following
\begin{theorem}\label{T: ses}
The system (\ref{A}) admits a strong exponential splitting with
projections sequence $P$ if and only if (\ref{A}) admits an
exponential splitting with respect to $P$ and $P$ is exponentially
bounded.
\end{theorem}
\begin{proof}
{\it Necessity.} We assume that system (\ref{A}) admits a (s.e.s.)
with respect to $P.$ Then, from (\ref{ses1}) for $m=n,$ it results
that $P$ is exponentially bounded. The implications
(\ref{ses1})$\Longrightarrow$(\ref{es1}), respectively
(\ref{ses2})$\Longrightarrow$(\ref{es2}) result by substitution of
$x$ with $P_nx$ in (\ref{ses1}) respectively of $x$ with $Q_mx$ in
(\ref{ses2}).

{\it Sufficiency.} If the projections sequence $P$ is exponentially
bounded then there exist $M,p\geq 1$ such that
$$\parallel P_n\parallel+\parallel Q_n\parallel\leq Mp^n,$$
for every $n\in\mathbb{N}.$ If system (\ref{A}) admits a (e.s.)
with respect to $P$ it results (via Theorem \ref{T: esinv}) that
there exist $0<a<b$ and $N,c\geq 1$ such that
\begin{align*}
\parallel{A}_m^nP_nx\parallel &\leq N a^{m-n}c^n\parallel
P_nx\parallel\leq MN a^{m-n}(pc)^n\parallel x\parallel\\
&=N_1 a^{m-n}c_1^n\parallel x\parallel
\end{align*}
and
$$b^{m-n}\parallel{B}_m^nQ_mx\parallel\leq Nc^m\parallel Q_mx\parallel\leq N_1c_1^m\parallel x\parallel$$
for all $(m,n,x)\in\Delta\times X,$ where $N_1=MN$ and $c_1=pc.$
\end{proof}
In the particular case when $c=1,$ we obtain
\begin{corollary}
The system (\ref{A}) admits uniform strong exponential splitting
with projections sequence $P$ if and only if (\ref{A}) admits
uniform exponential splitting with respect to $P$ and $P$ is
bounded.
\end{corollary}
\begin{remark}
If the system (\ref{A}) admits a strong exponential splitting then
it also admits an exponential splitting. The following example shows
that the converse is not true.
\end{remark}
\begin{example}\label{E: 2}
Let $X=\mathbb{R}^2$  endowed with the norm
\begin{equation*}
\parallel x\parallel=\max\{|x_1|,|x_2|\}|\;\;\text{for}\;\;x=(x_1,x_2)\in X.
\end{equation*}
Let $(P_n)$ be a sequence in $B(X)$ defined by
\begin{equation*}
P_n(x_1,x_2)=\left(x_1+\left(2^{n^2}-1\right)x_2,0\right).
\end{equation*}
It is a simple verification to see that $(P_n)$ is a projections
sequence with the complementary
\begin{equation*}
Q_n(x_1,x_2)=\left(\left(1-2^{n^2}\right)x_1,x_2\right).
\end{equation*}
Moreover,
\begin{equation*}
\parallel P_nx\parallel=2^{n^2}\parallel x\parallel,
\end{equation*}
\begin{equation*}
\parallel Q_nx\parallel=\left(2^{n^2}-1\right)|x_{2}|\leq\parallel Q_mx\parallel
\end{equation*}
and
\begin{equation*}
P_mP_n=P_n,\;\;Q_mQ_n=Q_m,\;\;Q_mP_n=0
\end{equation*}
for all $(m,n,x)\in\Delta\times X.$

We consider the linear discrete-time system (\ref{A}) defined by
the sequence $(A_n)$ given by
\begin{equation*}
A_n=2{P_n}+4Q_{n+1},\;\;\text{for all}\;n\in\mathbb{N}.
\end{equation*}
We observe that
\begin{equation*}
A_nP_n=P_{n+1}A_n=2{P_n},\;\;\text{for every}\;\;n\in\mathbb{N},
\end{equation*}
hence $(P_n)$ is invariant for (\ref{A}). The evolution operator
asociated to (\ref{A}) is
\begin{equation*}
{A}_m^n=2^{m-n}P_n+4^{m-n}Q_m,\;\;\text{for all}\;\;(m,n)\in\Delta.
\end{equation*}
We shall prove that $(P_n)$ is strongly invariant for (\ref{A}).

Let $(m,n)\in\Delta.$ In order to prove the injectivity of
${A}_m^n$ we consider
\begin{equation*}
x=Q_mz\in Ker\;P_n\;\;\text{with}\;\;{A}_m^nx=0.
\end{equation*}
Because
\begin{equation*}
0={A}_m^nx={A}_m^nQ_nz=4^{m-n}Q_mz,
\end{equation*}
it follows that $z\in Ker\;Q_m=Range\;P_m$ and hence
\begin{equation*}
x=Q_nz=Q_nP_mz=0.
\end{equation*}
To prove the surjectivity of ${A}_m^n$ from $Ker\;P_n$ to
$Ker\;P_m=Range\;Q_m,$ let $y=Q_mz\in Ker\;P_m.$ Then
\begin{equation*}
x=4^{n-m}Q_nz\in Ker\;P_n
\end{equation*}
with
\begin{equation*}
{A}_m^nx=4^{n-m}{A}_m^nQ_nz=Q_mz=y.
\end{equation*}
Thus $P=(P_n)$ is strongly invariant for system (\ref{A}) and the
skew-evolution operator associated to the pair
$(\text{\ref{A}},P)$ is
\begin{equation*}
{B}_m^nQ_m=4^{n-m}Q_n,\;\;\text{for all}\;\;(m,n)\in\Delta.
\end{equation*}
Furthermore, from
\begin{equation*}
\parallel{A}_m^nP_nx\parallel\leq 2^{m-n}\parallel P_nx\parallel
\end{equation*}
and
\begin{equation*}
4^{m-n}\parallel{B}_m^nQ_mx\parallel=\parallel
Q_nx\parallel\leq\parallel Q_mx\parallel
\end{equation*}
for all $(m,n,x)\in\Delta\times X,$ it results that (\ref{A}) admits
an (u.e.s.) (hence an (e.s.)) with respect to $(P_n).$

If we suppose that (\ref{A}) admits a (s.e.s.) with projections
sequence $(P_n)$ then, by Theorem \ref{T: ses}, it results that
$(P_n)$ is exponentially bounded, which is a contradiction because
\begin{equation*}
\parallel P_n\parallel=2^{n^2},\;\;\text{for
every}\;\;n\in\mathbb{N}.
\end{equation*}
\end{example}
\begin{remark}
 If the system (\ref{A}) admits a uniform
exponential splitting then it also admits an uniform exponential
dichotomy. The previous example shows that the converse implication
is not valid. More precise, if we suppose that system (\ref{A})
admits a uniform exponential dichotomy then there are two constants
$N\geq 1$ and $d\in (0,1)$ such that
$$\parallel A_m^nP_nx\parallel\leq Nd^{m-n}\parallel P_nx\parallel$$
for all $(m,n,x)\in\Delta\times X.$ In particular, for $m=2n$ we
have that
$$\left(\frac{2}{d}\right)^{n}\parallel\leq N$$
for all $n\in\Bbb{N},$ which is a contradiction.
\end{remark}
\begin{remark}
It is obvious that (u.e.s.)$\Longrightarrow$(e.s.). The following
example shows that the converse implication is not true.
\end{remark}
\begin{example}\label{E: 3}
Let $(P_n)$ be the projections sequence considered in Example
\ref{E: 2} and the linear discrete-time system (\ref{A}) defined by
the sequence $(A_n)$ given by
\begin{equation*}
A_n=2^{a_n-a_{n+1}}P_n+4^{a_{n+1}-a_n}Q_{n+1},
\end{equation*}
where
$$a_n=\frac{n}{1+2\cos^{2}\frac{n\pi}{2}},\;\;\text{for
all}\;\;n\in\mathbb{N}.$$

We have the evolution operator associated to (\ref{A})
\begin{equation*}
{A}_{m}^{n}=2^{a_n-a_m}P_n+4^{a_m-a_n}Q_m,\;\;\text{for
all}\;\;(m,n)\in\Delta
\end{equation*}
and respectively
\begin{equation*}
{B}_m^nQ_m=4^{a_n-a_m}Q_n,\;\;\text{for all}\;\;(m,n)\in\Delta
\end{equation*}
the skew-evolution operator associated to the pair
$(\text{\ref{A}},P).$ We observe that for all $(m,n)\in\Delta$ we
obtain
\begin{align*}
a_n-a_m
&=\frac{n-m}{3}+\frac{2n\sin^{2}\frac{n\pi}{2}}{3\left(1+2\cos^{2}\frac{n\pi}{2}\right)}-\frac{2m\sin^{2}\frac{m\pi}{2}}{3\left(1+2\cos^{2}\frac{m\pi}{2}\right)}\\
&\leq \frac{n-m}{3}+\frac{2n}{3}
\end{align*}
hence
\begin{equation*}
a_m-a_n\geq \frac{m-n}{3}-\frac{2n}{3}.
\end{equation*}
Then
\begin{equation*}
\parallel{A}_m^nP_nx\parallel=2^{a_n-a_m}\parallel
P_nx\parallel\leq 4^{\frac{2n}{3}}2^{-\frac{1}{3}(m-n)}\parallel
P_nx\parallel
\end{equation*}
and respectively
\begin{equation*}
4^{\frac{1}{3}(m-n)}\parallel{B}_m^nQ_mx\parallel\leq
4^{\frac{2}{3}m}\parallel Q_mx\parallel,
\end{equation*}
for all $(m,n,x)\in\Delta\times X.$ Finally, we observe that for
$N=1,$ $a=2^{-\frac{1}{3}},$ $b=4^{\frac{1}{3}}$ and
$c=4^{\frac{2}{3}}$ the system (\ref{A}) admits an (e.s.).

If we suppose that system (\ref{A}) admits an (u.e.s.) then there
exist the constants $N\geq 1,$ $\alpha\in\mathbb{R}$ such that
\begin{equation*}
2^{a_n-a_m}\parallel
P_nx\parallel=\parallel{A}_m^nP_nx\parallel\leq Ne^{\alpha
(m-n)}\parallel P_nx\parallel,
\end{equation*}
for all $(m,n,x)\in\Delta\times X.$ In particular, for $n=2k+1$ and
$m=n+1$ it follows that
\begin{equation*}
2^{\frac{4k+1}{3}}\parallel P_nx\parallel\leq Ne^{\alpha}\parallel
P_nx\parallel,
\end{equation*}
which is a contradiction.
\end{example}
\begin{remark}
 We observe that Example \ref{E: 2} shows that
(u.e.s.)$\nRightarrow$(s.e.s.). The following example presents a
system (\ref{A}) which admits (s.e.s.) and it does not admits
(u.e.s.).
\end{remark}
\begin{example}\label{E: 4}
Let $X=l^{\infty}(\mathbb{R})$ be the Banach space considered in
Example \ref{E: 2} and let $(P_n)$ be the projections sequence
defined by
$$
P_n(x_0,x_1,x_2,\ldots
)=\left(x_0+\left(2^{n}-1\right)x_1,0,x_2+\left(2^{n}-1\right)x_3,0,\ldots\right)
$$
with the complementary
$$
Q_n(x_0,x_1,x_2,\ldots
)=\left(\left(1-2^{n}\right)x_1,x_1,\left(1-2^{n}\right)x_3,x_3,\ldots\right).
$$
It is immediate to see that
$$
\parallel P_nx\parallel=2^{n}\parallel x\parallel\;\;\text{and}\;\;\parallel Q_n\parallel=\left(2^n-1\right)\sup\limits_{n\geq 0}|x_{2n+1}|\leq\parallel Q_mx\parallel
$$
for all $(m,n,x)\in\Delta\times X.$

Let (\ref{A}) be the linear discrete-time system defined by the
sequence
$$
A_n=2^{a_n-a_{n+1}}P_n+4^{a_{n+1}-a_n}Q_{n+1},
$$
where
$$a_n=\frac{n}{1+2\cos^{2}\frac{n\pi}{2}},\;\;\text{for
all}\;\;n\in\Bbb{N}.
$$
As in Example \ref{E: 3} it follows that $(P_n)$ is strongly
invariant for (\ref{A}) with
$$
{A}_{m}^{n}=2^{a_n-a_m}P_n+4^{a_m-a_n}Q_m,
$$
$$
{B}_m^nQ_m=4^{a_n-a_m}Q_n,\;\;\text{for all}\;\;(m,n)\in\Delta.
$$
Then
$$\parallel A_m^nP_nx\parallel\leq 4^{\frac{2n}{3}}2^{-\frac{1}{3}(m-n)}\parallel P_nx\parallel\leq 4^{\frac{5n}{3}}2^{-\frac{1}{3}(m-n)}\parallel x\parallel$$
and
$$4^{\frac{m-n}{3}}\parallel B_m^nQ_mx\parallel\leq 4^{\frac{2m}{3}}\parallel Q_mx\parallel\leq 4^{\frac{5m}{3}}\parallel x\parallel,$$
for all $(m,n,x)\in\Delta\times X.$

Thus (\ref{A}) admits a (s.e.s.) with respect to $(P_n).$ If we
suppose that (\ref{A}) admits a (u.e.s.) with respect to $(P_n)$
then there are $N\geq 1,$ $\alpha\in\mathbb{R}$ such that
$$2^{a_n-a_m}2^n\parallel x\parallel=2^{a_n-a_m}\parallel P_nx\parallel=\parallel A_m^nP_nx\parallel\leq Ne^{\alpha(m-n)}\parallel x\parallel$$
for all $(m,n,x)\in\Delta\times X.$ In particular for $m=2k+1$ and
$n=2k$ we obtain
$$4^{k+3}\leq Ne^{\alpha}$$
for every $k\in\Bbb{N},$ which is a contradiction.
\end{example}

\begin{remark}
The connections between the four splitting concepts considered in
this paper can be synthesized in the following diagram
\begin{eqnarray*}
(u.s.e.s.) & \Longrightarrow & (u.e.s.)\nonumber\\
\Downarrow  &  & \Downarrow\nonumber\\
(s.e.s.) & \Longrightarrow & (e.s.)\nonumber\\
\end{eqnarray*}
The presented examples shows that the implications
(s.e.s.)$\Rightarrow$(u.s.e.s.),\\  (e.s.)$\Rightarrow$(s.e.s.),
(e.s.)$\Rightarrow$(u.e.s.), (u.e.s.)$\Rightarrow$(u.s.e.s.),
(e.s.)$\Rightarrow$(u.s.e.s.),\\ (u.e.s.)$\Rightarrow$(s.e.s.) and
  (s.e.s.)$\Rightarrow$(u.e.s.) are not valid.

Finally, we obtained that the studied splitting concepts are
distinct. As a particular case, similar conclusions hold for the
dichotomy concepts defined in this paper.
\end{remark}

\section{Conclusion}
In this paper we consider three concepts of exponential splitting using two concepts of projections sequences: invariant and strongly invariant for general nonivertible and nonautonomous linear discrete-time systems in Banach spaces. These concepts are natural generalizations of some well-known concepts of dichotomies. Characterizations of these concepts of exponential splitting and connections (implications and counterexamples) between them are exposed.


\end{document}